\RequirePackage{fixltx2e}
\documentclass[oneside,english]{amsart}
\usepackage[T1]{fontenc}
\usepackage[latin9]{inputenc}
\usepackage{babel}
\usepackage{textcomp}
\usepackage{amstext}
\usepackage{amsthm}
\usepackage{amssymb}
\usepackage{enumerate}
\usepackage[all]{xy}
\usepackage[unicode=true,
 bookmarks=true,bookmarksnumbered=false,bookmarksopen=false,
 breaklinks=false,pdfborder={0 0 1},backref=false,colorlinks=false]
 {hyperref}
\hypersetup{pdftitle={Finite subgroups of the birational automorphism group are `almost' nilpotent of class at most two},
 pdfauthor={Attila Guld},
 pdfsubject={Algebraic Geometry}}

\makeatletter
\numberwithin{equation}{section}
\numberwithin{figure}{section}
\theoremstyle{plain}
\newtheorem{thm}{\protect\theoremname}
  \theoremstyle{plain}
  \newtheorem{prop}[thm]{\protect\propositionname}
  \theoremstyle{plain}
  \newtheorem{lem}[thm]{\protect\lemmaname}
  \theoremstyle{plain}
  \newtheorem{cor}[thm]{\protect\corollaryname}
  \theoremstyle{remark}
  \newtheorem{rem}[thm]{\protect\remarkname}
  \theoremstyle{definition}
  
  \theoremstyle{plain}
  
  \theoremstyle{definition}
  \newtheorem{defn}[thm]{\protect Definition}
\usepackage{fancyhdr}
\usepackage{ifpdf} 
\ifpdf 
 \IfFileExists{lmodern.sty}{\usepackage{lmodern}}{}
\fi

\makeatother

  \providecommand{\corollaryname}{Corollary}
  \providecommand{\examplename}{Example}
  \providecommand{\lemmaname}{Lemma}
  \providecommand{\propositionname}{Proposition}
  \providecommand{\questionname}{Question}
  \providecommand{\remarkname}{Remark}
  \providecommand{\theoremname}{Theorem}

\DeclareMathOperator{\Spec}{Spec}
\DeclareMathOperator{\Aut}{Aut}
\DeclareMathOperator{\GL}{GL}

\DeclareMathOperator{\Ker}{Ker}

\DeclareMathOperator{\Bir}{Bir}
\DeclareMathOperator{\KL}{\Gamma L}
\DeclareMathOperator{\CoH}{H}
\DeclareMathOperator{\Imag}{Im}
\DeclareMathOperator{\Cent}{C}

\DeclareMathOperator{\Z}{Z}
\DeclareMathOperator{\Diff}{Diff}
\DeclareMathOperator{\T}{T}

\DeclareMathOperator{\rk}{rk}
\begin{document}

\title[The birational automorphism group is nilpotently Jordan]{Finite subgroups of the birational automorphism group are `almost' nilpotent of class at most two}

\author{Attila Guld}
\email{guld.attila@renyi.mta.hu}

\thanks{The research was partly supported by the National Research,Development and Innovation Office (NKFIH) Grant No. K120697. 
The project leading to this application has received funding from the European Research Council (ERC) 
under the European Union's Horizon 2020 research and innovation programme (grant agreement No 741420).}

\address{
R\'enyi Alfr\'ed Matematikai Kutat\'oint\'ezet\\
Re\'altanoda utca 13-15.\\
Budapest, H1053\\
Hungary}

\begin{abstract}
We call a group $G$ nilpotently Jordan of class at most $c$ $(c\in\mathbb{N})$ 
if there exists a constant $J\in\mathbb{Z}^+$ such that every finite subgroup $H\leqq G$ contains a nilpotent subgroup $K\leqq H$ of class at most $c$ and index at most $J$.\\
We show that the birational automorphism group of a variety over a field of characteristic zero is nilpotently Jordan of class at most two. 
\end{abstract}

\keywords{birational automorphism group, birational selfmap, nilpotent group, Jordan group}
\maketitle

\section{Introduction}

\begin{defn}
\label{nilpJord}
A group $G$ is called Jordan, solvably Jordan or nilpotently Jordan of class at most $c$ ($c\in\mathbb{N}$) if there exists a constant $J\in\mathbb{Z}^+$ such that every finite subgroup $H\leqq G$ has a subgroup $K\leqq H$ such that
$|H:K|\leqq J$ and $K$ is Abelian, solvable or nilpotent of class at most $c$, respectively.
\end{defn}

\begin{thm}
\label{main}
The birational automorphism group of a variety over a field of characteristic zero is nilpotently Jordan of class at most two.
\end{thm}

\subsection{History}
In the following discussion we shortly sketch the history of the question of Jordan type properties in birational geometry over fields of \textit{characteristic zero}.\\ 
Research about investigating the Jordan property of the birational automorphism group of a variety was initiated by J.-P. Serre (\cite{Se09}) and V. L. Popov (\cite{Po11}).
 In \cite{Se09} J.-P. Serre settled the problem for the Cremona group of rank two, while  in the articles \cite{Po11},\cite{Za15} V. L. Popov and Yu. G. Zarhin solved the question for one and two dimensional varieties.
They found that the birational automorphism group of a curve or a surface is Jordan, save when the variety is birational to a direct product of an elliptic curve and the projective line. 
This later case was examined in \cite{Za15}, where -based on calculations of D. Mumford- the author was able to conclude that the birational automorphism group contains Heisenberg $p$-groups for arbitrarily large prime numbers $p$. 
Hence it does not enjoy the Jordan property.\\ 

In \cite{PS14} and \cite{PS16} Yu. Prokhorov and C. Shramov made major contributions to the subject using the arsenal of the Minimal Model Program (MMP)
 and assuming the result of the Borisov-Alexeev-Borisov (BAB) conjecture 
(which has later been verified in the celebrated article \cite{Bi16} of C. Birkar;  for a survey paper on the work of C. Birkar and its connection to the Jordan property the interested reader can consult with \cite{Ke19}). 
Amongst many interesting results Yu. Prokhorov and C. Shramov proved that the birational automorphism group of a rationally connected variety and 
the birational automorphism group of a non-uniruled variety is Jordan. To answer a question of D. Allcock, they also introduced the concept of solvably Jordan groups, 
and showed that the birational automorphism group of an arbitrary variety is solvably Jordan.\\

The landscape is strikingly similar in differential geometry. The techniques are fairly different, still the results converge to similar directions. 
In the followings we briefly review the history of the question of Jordan type properties of diffeomorphism groups of smooth compact real manifolds. (We note that there are many other interesting setups which were considered by differential geometers; 
for a very detailed account see the Introduction of \cite{MR18}.) As mentioned in \cite{MR18}, during the mid-nineties  \'E. Ghys conjectured that the diffeomorphism group of a smooth compact real manifold is Jordan, 
and he proposed this problem in many of his talks (\cite{Gh97}). The case of surfaces follows from the Riemann-Hurwitz formula (see \cite{MR10}), the case of 3-folds are more involved. 
In \cite{Zi14} B. P. Zimmermann proved the conjecture for them using the geometrization of compact 3-folds (which follows from the work of W. P. Thurston and G. Perelman). 
I. Mundet i Riera also verified the conjecture for several interesting cases, like tori, projective spaces, homology spheres and manifolds without zero Euler characteristic (\cite{MR10},\cite{MR16}, \cite{MR18}).\\ 
However, in 2014, B. Csik\'os, L. Pyber and E. Szab\'o found a counterexample (\cite{CPS14}). 
Their construction was remarkably analogous to the one of Yu. G. Zarhin. They showed that if the manifold $M$ is diffeomorphic to the direct product of the two-sphere and the two-torus 
or to the total space of any other smooth orientable two-sphere bundle over the two-torus, then the diffeomorphism group contains Heisenberg $p$-groups for arbitrary large prime numbers $p$. Hence $\Diff(M)$ cannot be Jordan.
As a consequence, \'E. Ghys improved on his previous conjecture and proposed the problem of showing that the diffeomorphism group of a compact real manifold is nilpotently Jordan (\cite{Gh15}). 
As the first trace of evidence,  I. Mundet i Riera and  C. Sa\'ez-Calvo showed that the diffeomorphism group of a 4-fold is nilpotently Jordan of class at most 2 (\cite{MRSC19}). 
Their proof uses results form the classification theorem of finite simple groups.\\

In summary, we have seen that the birational automorphism group of many varieties are Jordan, however
even amongst surfaces one can find a counterexample. On the other hand, if we replace the Abelian
property with the slightly weaker solvability property, then the birational automorphism group
of every variety enjoys the solvably Jordan property. The picture is fairly similar in differential geometry as well.
These facts naturally raise the question that how much we can strengthen the condition of solvability.\\
Motivated by these antecedents, in this article, we investigate the nilpotenly Jordan property for birational automorphism groups of varieties.
 
\subsection{Description of the proof}

In this section we briefly sketch the main ideas of the proof.\\

At various points of the proof we are allowed to make birational modifications. We can resolve singularities of varieties and we can resolve indeterminacies of rational maps.
We are also allowed to replace a finite group (which acts on a variety) by one of its subgroups of bounded index (if the bound on the index only depends on the birational class of the variety).\\

Here we present an argument which omits technical issues, hence it is not precise, however captures the core ideas behind our proof. 
(Here we will not include the steps which replace a variety by one of its birational models, 
nor the steps which replace the finite transformation group with one of its bounded index subgroups.)\\
Our argument is heavily inspired by the work of Yu. Prokhorov and C. Shramov on the Jordan property of the birational automorphism groups of rationally connected varieties (\cite{PS16}).
To be able to handle arbitrary varieties, we will consider the MRC fibration (as the initial data) and carry out a relative version of the proof which can be found in \cite{PS16}.\\
Let $X$ be a variety and $G$ be a finite subgroup of its birational automorphism group. Without loss of generality we can assume that $X$ is a complex variety (Lemma \ref{C}), 
furthermore by regularization of finite group actions on varieties, we can assume that $G$ is a finite subgroup of the biregular automorphism group (Lemma \ref{reg}).\\
In the rationally connected case Yu. Prokhorov and C. Shramov were able to find a vector space on which $G$ acts faithfully (via linear automorphisms). Than Jordan's theorem on linear groups finished their proof.
In the general case, we will look for a vector bundle on which $G$ acts faithfully via vector bundle automorphisms in such a way that the $G$-action descends to an Abelian group action on the base variety of the vector bundle (Theorem \ref{VB}).
Hence we can consider the generic fibre of the vector bundle. It is a vector space over the function field of the base variety, on which $G$ acts faithfully via semilinear automorphisms. 
Than we can finish our proof by a Jordan type theorem on semilinear groups (Theorem \ref{groupmain}).\\
Now we investigate how to find the aforementioned vector space and vector bundle.
By the help of the MMP, boundedness of Fano varieties and a technique 
to find rationally connected fibrations (which was developed by C. D. Hacon and J. McKernan in \cite{HM07}) 
Yu. Prokhorov and C. Shramov found a proper (strictly smaller dimensional) $G$-invariant rationally connected closed subvariety $Y\subsetneqq X$.
Iterating this step, they found a $G$-fixed point (a zero dimensional $G$-invariant rationally connected variety) $P\in X$. The tangent space $\T_PX$ provides the desired vector space on which $G$ acts faithfully.
In the general case, let $\phi:X\to Z$ be the MRC fibration (Theorem \ref{MRC}). By the functoriality of the MRC fibration, there is an induced $G$-action on $Z$, which makes $\phi$ $G$-equivariant (Corollary \ref{GMRC}). 
The base $Z$ is non-uniruled, hence its birational automorphism group is Jordan by another important theorem of Yu. Prokhorov and C. Shramov (\cite{PS14}).  
So we can assume that the $G$-action on $Z$ is Abelian. By the MMP, boundedness of Fano varieties and the technique to find rationally connected fibrations,
we are able to find a proper  $G$-invariant closed subvariety $Y\subsetneqq X$ with the following property.   
There exist a variety $W$ endowed with an Abelian $G$-action (where the $G$-action on $W$ derives from the $G$-action on $Z$), 
and a $G$-equivariant dominant morphism $Y\to W$ which has rationally connected general fibres (Lemma \ref{PullBackF}).
Iterating this step (Lemma \ref{Bundle1}), we find a proper $G$-invariant subvariety $Y_0\subsetneqq X$, a variety $W_0$ endowed with an Abelian $G$-action 
and a $G$-equivariant morphism $Y_0\to W_0$ such that the general fibres of the morphism are zero dimensional rationally connected varieties, i.e. rational points. Hence $Y_0$ is birational to $W_0$, hence the $G$-action on $Y_0$ is Abelian.
Taking the normal bundle of $Y_0\subsetneqq X$ results the vector bundle which we look for (Lemma \ref{MainLemma}).\\
Going backwards in the proof, now we sketch how to find the proper subvariety $Y$. As usual let us first study the case when $X$ is rationally connected. 
By the help of the MMP we can find a $G$-equivariant dominant rational map $h:X\dashrightarrow F$, where $F$ is a positive dimensional Fano variety and $h$ has rationally connected fibres,
moreover $h$ can be written as a composite of steps of the MMP, i.e. as a composite of divisorial contractions, flips and Mori fibrations. 
Using boundedness of Fano varieties and Jordan's theorem on linear groups, we can find a $G$-invariant fixed point $Q\in F$. By the technique which finds rationally connected fibrations
we can pull back this zero dimensional proper rationally connected closed subvariety $Q\in F$ to a proper $G$-invariant rationally connected closed subvariety $Y\subsetneqq X$.  
Now consider the general case. Let $\phi:X\to Z$ be the MRC fibration. The generic fibre of $\phi$ is rationally connected and the base $Z$ is non-uniruled. Therefore both of them has Jordan birational  automorphism groups.
Hence we can assume that, there is a normal Abelian subgroup $N\unlhd G$ such that the quotient group $G/N$ is Abelian, and $N$ acts trivially on $Z$ (i.e. the $G$-action on $Z$ descends to the Abelian $G/N$-action). 
(See Theorem \ref{AbyA}.)
By the help of the MMP we can find a $G$-equivariant  Mori fibration $F\to B$ such that $N$ acts trivially on the base $B$ and non-trivially on the total space $F$ (Lemma \ref{F}).
Moreover there is a dominant rational map $h:X\dashrightarrow F$, with rationally connected fibres,
which can be written as a composite of steps of the MMP, i.e. as a composite of divisorial contractions, flips and Mori fibrations. 
Using boundedness of Fano varieties and Jordan's theorem on linear groups, we can find a $G$-invariant proper subvariety $W\subsetneqq F$ on which $N$ acts trivially (Lemma \ref{W}).
By the technique 
which finds rationally connected fibrations  we can find a proper closed $G$-invariant subvariety $Y\subsetneqq X$ 
and a $G$-equivariant dominant morphism $Y\to W$ which has rationally connected general fibres (Lemma \ref{PullBackF}). This finishes the proof.\\

In the actual proof there are some technical difficulties which we need to handle. Because of the MMP we lose smoothness of varieties, however taking the normal bundle works for smooth closed subvarieties of smooth ambient varieties.
Also because of the MMP we only get rational maps instead of morphisms. We can solve these issues by frequently resolving singularities of varieties and indeterminacies of rational maps. 
To avoid some of these problems, we do not require $Y$ to be a smooth subvariety of $X$. Instead we will consider varieties and vector bundles over them, where $G$ acts faithfully on the vector bundle (Lemma \ref{Bundle1}).
This models well enough normal bundles for our purposes.

\subsection{Structure of the article}
The article is organized in the following way. In Section \ref{P} we recall some definitions and theorems, which are essential to carry out our proof. 
In Section \ref{FGV} we collect and prove results about finite group actions on varieties. 
Section \ref{MRCtoVB} is the main body of the paper, we describe the construction of the vector bundle from the MRC fibration.
Section \ref{JSL} deals with the proof of the Jordan-type theorem on semilinear groups. 
Finally, in Section \ref{PMT} we prove our main theorem.

\subsection{Acknowledgements}
The author is very grateful to E. Szab\'o for many helpful discussions. The author would like to thank C. Shramov for drawing his attention to Chermak-Delgado theory.

\section{Preliminaries}
\label{P}

\subsection{Nilpotent groups}
\label{NG}
We recall the definition of nilpotent groups and some of the basic properties of nilpotent groups of class at most two.

\begin{defn}
Let $G$ be a group.
Let $\Z_0(G)=1$ and define $\Z_{i+1}(G)$ as the preimage of $\Z(G/\Z_i(G))$ under the natural quotient group homomorphism $G\to G/\Z_i(G)$ $(i\in\mathbb{N})$. The series of groups
$1=\Z_0(G)\leqq\Z_1(G)\leqq\Z_2(G)\leqq...$ is called the upper central series of $G$.\\
Let $\gamma_0(G)=G$ and let $\gamma_{i+1}(G)=[\gamma_i(G),G]$ ($i\in\mathbb{N}$, and $[,]$ denotes the commutator operation). The series of groups
$G=\gamma_0(G)\geqq\gamma_1(G)\geqq\gamma_2(G)\geqq...$ is called the lower central series of $G$.\\
$G$ is called nilpotent if one (hence both) of the following equivalent conditions hold:
\begin{itemize}
\item
There exists $m\in\mathbb{N}$ such that $\Z_m(G)=G$.
\item
There exists $n\in\mathbb{N}$ such that $\gamma_n(G)=1$.
\end{itemize} 
If $G$ is a nontrivial nilpotent group, then there exists a natural number $c$ for which  $\Z_c(G)=G$, $\Z_{c-1}(G)\neq G$ and  $\gamma_c(G)=1$, $\gamma_{c-1}(G)\neq 1$ holds. $c$ is called the nilpotency class of $G$. 
(If  $G$ is trivial, then its nilpotency class is zero.)
\end{defn}

\begin{rem}
Note that $\Z_1(G)$ is the centre of the group $G$, while $\gamma_1(G)$, also denoted by $G'$ or by $[G,G]$, is the commutator subgroup. A non-trivial group $G$ is nilpotent of class one if and only if it is Abelian.\\
Nilpotency is the property between the Abelian and the solvable properties. The Abelian property implies nilpotency, while nilpotency implies solvability.
\end{rem}

\begin{rem}
Typical examples of nilpotent groups are finite $p$-groups (where $p$ is a prime number). If we restrict our attention to finite nilpotent groups, even more can be said. 
(Recall that a Sylow $p$-subgroup of a finite group is a maximal $p$-group contained in the group.)
A finite group is nilpotent if and only if it is the direct product of its Sylow subgroups (Theorem 6.12 in \cite{CR62}).
\end{rem}

The following propositions describes some features of nilpotent groups of class at most two. 
\begin{prop}
\label{CE}
A group is nilpotent of class at most two if and only if it is a central extension of two Abelian groups.
\end{prop}

\begin{prop}
\label{ICmap}
Let $G$ be a nilpotent group of class at most two. 
Let $g\in G$ be an arbitrary fixed element.
The maps $\varphi_g$ and $\psi_g$ defined by the help of the commutators with the fixed element $g$
\begin{gather*}
\varphi_g:G\to [G,G]\\
x\mapsto[x, g]\\
\psi_g:G\to [G,G]\\
x\mapsto[x, g]
\end{gather*} 
give group homomorphisms.
\end{prop}

\subsection{Reduction to the field of the complex numbers}

We show that it is enough to prove Theorem \ref{main} over the field of the complex numbers.

\begin{lem}
\label{C}
It is enough to prove Theorem \ref{main} over the field of the complex numbers.
\end{lem}

\begin{proof}
Let $k$ be a field of characteristic zero and $X$ be a variety over $k$. First assume that $X$ is geometrically integral.
We can fix a finitely generated field extension $l_0|\mathbb{Q}$ and an $l_0$-variety $X_0$ such that $X\cong X_0\times_{l_0}\Spec k$. 
Fix a field embedding $l_0\hookrightarrow\mathbb{C}$ and let $X^*\cong X_0\times_{l_0}\Spec\mathbb{C}$.
For an arbitrary finite subgroup $G\leqq\Bir(X)$ we can find a finitely generated field extension $l_1|l_0$ such that the elements of $G$ can be defined as birational transformations over the field $l_1$. Hence
 $G\leqq \Bir(X_1)$, where $X_1\cong X_0\times_{l_0}\Spec l_1$. 
We can extend the fixed field embedding $l_0\hookrightarrow\mathbb{C}$ to a field embedding $l_1\hookrightarrow\mathbb{C}$.
Therefore $X^*\cong X_0\times_{l_0}\Spec\mathbb{C}\cong X_1\times_{l_1}\Spec\mathbb{C}$, and we can embed $G$ into the birational automorphism group of the complex variety $X^*$.
As the birational class of the complex variety $X^*$ only depends on the birational class of the variety $X$, it is enough to examine complex varieties. This proves the lemma when $X$ is geometrically integral.\\
If $X$ is not geometrically integral, then it is still geometrically reduced as we work in characteristic zero. We can still construct the $l_0$-variety $X_0$, the $l_1$-scheme $X_1$ and the complex scheme $X^*$.
Therefore $G\leqq\Bir(X)$ embeds into $\Bir(X^*)$ just like in the geometrically integral case.
There exists a constant only depending on the birational class of $X$ such that $X^*$ has $C$ many irreducible components. Therefore a finite subgroup $H\leqq G$ of index at most $C!$
leaves all irreducible components of $X^*$ invariant. Hence $H$ has a nilpotent subgroup of class at most two of bounded index by the complex case. 
The bound on the index only depends on the birational classes of the components of $X^*$, hence it only depends on the birational class of $X$. 
Therefore $G$ has a nilpotent subgroup of class at most two of bounded index, where the bound on the index only depends on the birational class of $X$. This finishes the proof.
\end{proof}

\subsection{Maximal Rationally Connected fibration}
We recall the concept of the maximal rationally connected fibration. For a detailed treatment see Chapter $4$ of \cite{Ko96}, for the non-uniruledness of the basis see Corollary 1.4 in \cite{GHS03}.

\begin{thm}
\label{MRC}
Let $X$ be a smooth proper complex variety. The pair $(Z,\phi)$ is called the maximal rationally connected (MRC) fibration if 
\begin{itemize}
\item $Z$ is a complex variety,
\item $\phi:X\dashrightarrow Z$ is a dominant rational map,
\item there exist open subvarieties $X_0$ of $X$ and $Z_0$ of $Z$ such that $\phi$ descends to a proper morphism between them $\phi_0:X_0\to Z_0$ with rationally connected fibres,
\item if $(W,\psi)$ is another pair satisfying the three properties above, then $\phi$ can be factorized through $\psi$. More precisely, there exists a rational map $\tau: W\dashrightarrow Z$ such that $\phi=\tau\circ\psi$. 
\end{itemize}
The MRC fibration exists and is unique up to birational equivalence. Moreover the basis $Z$ is non-uniruled.
\end{thm}

The MRC fibration is functorial, that is the content of the following theorem (Theorem 5.5 of Chapter 4 in \cite{Ko96}).
\begin{thm}
\label{FunctMRC}
Let $X_1$ and $X_2$ be smooth proper complex varieties. Let the pairs $(Z_1,\phi_1)$ and $(Z_2,\phi_2)$ be the corresponding MRC fibrations.
If $f:X_1\dashrightarrow X_2$ is a dominant rational map, then there exists a dominant rational map $g:Z_1\dashrightarrow Z_2$ making the diagram below commutative.
\[
\xymatrix{
X_1 \ar@{-->}[r]^f \ar@{-->}[d]^\phi & X_2\ar@{-->}[d]^{\phi_2}\\
Z_1 \ar@{-->}[r]^g & Z_2
}
\]
\end{thm}

The following claim is an immediate corollary. It states that, if $X$ is endowed with a $G$-action, then there is an induced $G$-action on $Z$.
\begin{cor}
\label{GMRC}
Let $X$ be a smooth proper complex variety and $(Z,\phi)$ be the corresponding MRC fibration. 
The birational automorphism group $\Bir(X)$ acts on $Z$ via birational automorphisms and $\phi$ is $\Bir(X)$-equivariant. 
\end{cor}

\subsection{Regularization}

The next lemma is a slight extension of the well-known \linebreak (smooth) regularization of finite group actions on varieties  (Lemma-Definition 3.1 in \cite{PS14}).
Because of this lemma, we can study finite birational group actions by the help of finite biregular group actions.

\begin{lem}
\label{reg}
Let 
\begin{itemize} 
\item $G$ be a finite group,
\item $X$ and $Z$  be projective varieties over a field of characteristic zero endowed with (not necessarily faithful) $G$-actions via birational automorphisms,
\item $\phi:X\dashrightarrow Z$ be a dominant $G$-equivariant rational map,
\end{itemize}
then there exist
\begin{itemize} 
\item  smooth projective varieties $X^*$ and $Z^*$  endowed with (not necessarily faithful) $G$-actions via biregular automorphisms, such that
\item $X^*$ is $G$-equivariantly birational to $X$,
\item $Z^*$ is $G$-equivariantly birational to $Z$, moreover there exists
\item a dominant $G$-equivariant morphism $\phi^*:X^*\to Z^*$  , such that the $G$-equivariant diagram below is commutative.
\[
\xymatrix{
X \ar@{-->}[r]^\sim \ar@{-->}[d]^\phi & X^* \ar[d]^{\phi^*}\\
Z \ar@{-->}[r]^\sim & Z^*
}
\]
\end{itemize}
\end{lem}

\begin{proof}
Let $k(Z)\leqq k(X)$ be the field extension corresponding to the function fields of $Z$ and $X$, induced by $\phi$. 
Take the induced $G$-action on this field extension and let $k(Z)^G\leqq k(X)^G$ be the field extension of the $G$-invariant elements. 
Consider a projective model of it, i.e. let $X_1$ and $Z_1$ be projective varieties such that $k(X_1)\cong k(X)^G$ and $k(Z_1)\cong k(Z)^G$, 
furthermore let $\varrho_1: X_1\to Z_1$ be a (projective) morphism inducing the field extension  $k(Z_1)\cong k(Z)^G\leqq k(X)^G\cong k(X_1)$.
By normalizing $X_1$ in the function field $k(X)$ and $Z_1$ in the function field $k(Z)$ we get projective varieties $X_2$ and $Z_2$, moreover $\varrho_1$ induces a $G$-equivariant morphism $\varrho_2:X_2\to Z_2$ between them.\\
As the next step, we can take a $G$-equivariant resolution of singularities $\widetilde{Z_2}\to Z_2$. After replacing $Z_2$ by $\widetilde{Z_2}$ 
we can $G$-equivariantly resolve the indeterminacies of $X_2\dashrightarrow Z_2$, hence we can assume that $Z_2$ is smooth.  Hence $G$-equivariantly resolving the singularities of $X_2$ finishes the proof.
\end{proof}

\subsection{The Minimal Model Program}
Applying the results of the famous article by C. Birkar, P. Cascini, C. D. Hacon and J. McKernan (\cite{BCHM10}) enables us to use the arsenal of the Minimal Model Program. 
As a consequence, we can examine rationally connected fibrations by the help of Fano fibrations. 
For the later we can use boundedness results because of yet another famous theorem  by C. Birkar (\cite{Bi16}). (This theorem was previously known as the BAB Conjecture).\\

Let $X$ be a variety and $G$ be a group acting on $X$ via biregular automorphisms. Recall that $X$ is called $G\mathbb{Q}$-factorial, 
if every $G$-invariant $\mathbb{Q}$-divisor on $X$ is $\mathbb{Q}$-Cartier. 

\begin{thm}
\label{MMP}
Let 
\begin{itemize}
\item $G$ be a finite group,
\item $X$ be a smooth projective complex variety endowed with a (not necessarily faithful) $G$-action via biregular automorphisms, 
\item $Z$ be a normal projective complex variety endowed with a (not necessarily faithful) $G$-action via biregular automorphisms, 
\item let $\dim X>\dim Z$, and
\item let $\phi:X\to Z$ be a $G$-equivariant dominant morphism with rationally connected general fibres. 
\end{itemize}
Then, we can run a $G$-equivariant Minimal Model Program (MMP)
on $X$ relative to $Z$ which results a Mori fibre space. In particular, the Minimal Model Program gives a $G$-equivariant commutative diagram 
\[
\xymatrix{
X\ar@ {-->} [r]^{\sim} \ar[rd]^{\phi} & F \ar[r] \ar[d] & B \ar[ld]\\
& Z
}
\]
\begin{itemize}
\item where $F$ and $B$ are normal projective complex varieties endowed with (not necessarily faithful) $G$-actions via biregular automorphisms,
\item $F$ is $G\mathbb{Q}$-factorial,
\item $F$ has terminal singularities,
\item $X\overset{\sim}{\dashrightarrow}F$ is a $G$-equivariant birational automorphism,
\item $\dim X=\dim F>\dim B$,
\item the general fibres of $F\to B$ are Fano varieties with terminal singularities.
\end{itemize}
\end{thm}
\begin{proof}
By Corollary 1.3.3 of \cite{BCHM10}, we can run a relative MMP on $\phi:X\to Z$ (which results a Mori fibre space) if the canonical divisor of $X$ is not $\phi$-pseudo-effective. It can be done equivariantly if we have 
finite group actions. (See Section 2.2 in \cite{KM98} and Section 4 of \cite{PS14} for further discussions on the topic.)\\
By generic smoothness, a general fibre of $\phi$ is positive dimensional smooth rationally connected projective complex variety. 
Therefore if $x$ is a general closed point of a general fibre $F$, then there exists a rational curve $C_x$ running through $x$, lying entirely in the fibre $F$, which is a free rational curve of the variety $F$ (Theorem 1.9 of Chapter 4 in \cite{Ko96}). 
Since $C_x$ is a free rational curve $C_x.K_F\leqq-2$. As $F$ is a general fibre, by adjunction, we have $C_x.K_X=C_x.K_F\leqq -2$.
Since the inequality holds for every general closed point of every general fibre, $K_X$ cannot be $\phi$-pseudo-effective.
\end{proof}

If $X$ and $Z$ are complex varieties (possibly with mild singularities) such that $\dim X=\dim Z$, and $f:X\to Z$ is a dominant morphism between them with rationally connected general fibres, 
then the general fibres must be rational points, hence $X$ and $Z$ are birational.
In this case, we can also run a relative MMP on $X$ over $Z$. It results a variety which is isomorphic to $Z$.

\begin{lem}
\label{small}
Let
\begin{itemize}
\item $X$ and $Z$ be normal projective complex varieties, such that
\item $Z$ has terminal singularities,
\item let $f:X\to Z$ be a birational morphism, such that
\item the canonical divisor $K_X$ is $f$-nef.
\end{itemize}
Then $f$ is small, i.e. its exceptional locus has codimension at least two.
\end{lem}

\begin{proof}
Let $E_1, E_2,..., E_n$ be the exceptional prime divisors of $f$.
Since $Z$ has terminal singularities
$K_X=f^*K_Z+\sum_{i=1}^n a_i E_i$,
where $K_Z$ is the canonical divisor of $Z$ (such that $f_*K_X=K_Z$), and $a_i\in\mathbb{R}^+$ ($\forall 1\leqq i\leqq n$).
Let $E=\sum_{i=1}^n a_i E_i$. Clearly $0\leqq E$.\\ 
Let $C\subseteqq X$ be an arbitrary curve contracted by $f$. Then
$0\leqq K_X.C=f^*K_Z.C+E.C=E.C$, as $K_X$ is nef over $Z$. Let $D=-E$.
By the above inequality, $-D$ is $f$-nef, moreover $f_*(D)=0$. Hence by the Negativity Lemma (Lemma 4.15 in \cite{Bi12}) $0\leqq D=-E$.
Therefore $E=0$, so $f$ has no exceptional prime divisors. In other words, $f$ is small.
\end{proof}

\begin{thm}
\label{MMP2}
Let 
\begin{itemize}
\item $G$ be a finite group,
\item $X$ be a smooth projective complex variety endowed with a (not necessarily faithful) $G$-action via biregular automorphisms, 
\item $Z$ be a normal projective complex variety endowed with a (not necessarily faithful) $G$-action via biregular automorphisms, such that
\item $Z$ is $G\mathbb{Q}$-factorial and has terminal singularities, moreover
\item let $\phi:X\to Z$ be a $G$-equivariant birational morphism. 
\end{itemize}
Then either $\phi:X\to Z$ is an isomorphism, or we can run a $G$-equivariant Minimal Model Program (MMP)
on $X$ over $Z$. The MMP factorizes $\phi$ into a sequence of $G$-equivariant flips and divisorial contractions. 
More precisely, the MMP results a $G$-equivariant commutative diagram
\[
\xymatrix{
X=X_0\ar@ {-->} [r]^{\sim} \ar[rrd]^{\phi}   &   X_1\ar@ {-->} [r]^{\sim} \ar[rd]    &   ...\ar@ {-->} [r]^{\sim} \ar[d]   &     X_{n-1}\ar@ {-->} [r]^{\sim}  \ar[ld]    &     X_n \ar[lld]_{\psi} \\
                                                                        &                                                            &     Z,                                                                                                  
}
\]
where the birational maps are either flips or divisorial contractions and $\psi:X_n\to Z$ is an isomorphism.
\end{thm}
\begin{proof}
If $\phi$ is an isomorphism, then the proof is finished, so assume otherwise.
By a $G$-equivariant version of Corollary 1.4.2 of \cite{BCHM10}, we can run a relative $G$-equivariant MMP on $X$ over $Z$.\\
Clearly, it results a diagram of the form described by the theorem. We only need to prove that $\psi$ is an isomorphism. 
$K_{X_n}$ is $\psi$-nef as the MMP terminates at $X_n$. By Lemma \ref{small}, $\psi$ is small. Hence $\psi$ must be an isomorphism
as $Z$ is $G\mathbb{Q}$-factorial (and $X$ and $Z$ are projective). This finishes the proof. 
\end{proof}

\subsection{Boundedness of Fano varieties}
Fano varieties with mild singularities form a bounded family. That was the long standing Borisov-Alekseev-Borisov conjecture,
which has been proved recently in the famous article by C. Birkar (Theorem 1.1 in \cite{Bi16}).

\begin{thm}
\label {BirFam}
Let $d$ be a non-negative integer, and $\epsilon$ be a positive real number. 
The collection formed by those at most $d$-dimensional complex projective varieties $X$ for which there exists an $\mathbb{R}$-boundary divisor $B$ such that
\begin{itemize}
\item $(X,B)$ be is an $\epsilon$-logcanonical pair, 
\item $-(K_X+B)$ is nef and big
\end{itemize}
is a bounded family.
\end{thm} 

In this article we will use the following corollary of this theorem. 
(Recall that terminal singularities can be defined for varieties over arbitrary fields of characteristic zero the same way as they are defined for complex varieties.)

\begin{cor}
\label{Fam2}
Let $d$ be a non-negative integer. There exist positive integers $m=m(d), N=N(d)$ and $E=E(d)$, only depending on $d$, such that if
\begin{itemize}
\item $k$ is an arbitrary field of characteristic zero,
\item $F$ is an arbitrary at most $d$ dimensional Fano variety over the ground field $k$ with terminal singularities,
\end{itemize}
then the following claims hold
\begin{itemize}
\item $-mK_F$ is very ample,
\item $\dim_k\CoH^0(F,-mK_F)\leqq N$, 
\item the degree of the closed subvariety $F\subseteqq\mathbb{P}(\CoH^0 (X,-mK_F)^*)$ is at most E.
\end{itemize}
\end{cor}

\begin{proof}
Fix $k$ and $F$ with the properties described by the theorem. There exist a finitely generated field extension $l_0|\mathbb{Q}$ and a  Fano variety $F_0$ over $l_0$ such that $F\cong F_0\times_{l_0}\Spec k$.
Consider an embedding of fields $l_0\hookrightarrow\mathbb{C}$, and let $F_1\cong F_0\times_{l_0}\Spec \mathbb{C}$.
Since complex Fano varieties with terminal singularities of bounded dimension form a bounded family (Theorem\ref{BirFam}), we can bound all of their numerical invariants simultaneously.
Therefore there exist constants $m=m(d),N=N(d), E=E(d)\in\mathbb{N}$, only depending on $d$,
such that $m$-th power of the anticanonical divisor embeds $F_1$ into the at most $(N-1)$ dimensional complex projective space $\mathbb{P}(\CoH^0 (X,-mK_{F_1})^*)$ with degree at most $E$.
Since the $m$-th power of the anticanonical divisor is defined over any field, this embedding is defined over any field, in particularly over $k$. 
So $F\hookrightarrow \mathbb{P}(\CoH^0 (X,-mK_F)^*)$ is a closed embedding of degree at most $E$.
This, in particular, implies that $-mK_F$ is very ample, $\dim_k\CoH^0 (X,-mK_F)\leqq N$ and that the degree of the closed subvariety $F\subseteqq\mathbb{P}(\CoH^0 (X,-mK_F)^*)$ is at most E.
This finishes the proof.
\end{proof}

\section{Finite group actions on varieties}
\label{FGV}

\subsection{Jordan type properties of birational automorphism groups of varieties}
In this subsection we collect results about Jordan properties of birational automorphism groups of rationally connected varieties and non-uniruled varieties.\\

The following theorem is amongst the most important results of the field. It is due to Yu. Prokhorov and C. Shramov (Theorem 1.8 in \cite{PS14} and Theorem 1.8 in \cite{PS16}). 

\begin{thm}
\label{nurc}
The birational automorphism groups of non-uniruled varieties and rationally connected varieties (over fields of characteristic zero) enjoy the Jordan property.
\end{thm}

In case of rationally connected varieties we can formulate an even stronger theorem. 
We can find Jordan constants for all at most $e$ dimensional rationally connected varieties simultaneously (Theorem 1.8 in \cite{PS16}).
We also need to find bounded index \textit{characteristic} Abelian subgroups. We can achieve it by an easy application of the beautiful Chermak-Delgado theory (Theorem 1.41 in \cite{Is08}).

\begin{thm}
\label{CD}
Let $G$ be a finite group. There exists a characteristic Abelian subgroup $M\leqq G$ such that for an arbitrary Abelian subgroup $A\leqq G$,
we have the following inequality of indices, $|G:M|\leqq|G:A|^2$.  
\end{thm}

\begin{thm}
\label{RCJor}
Let $e$ be a non-negative integer. There exists a constant $R=R(e)\in\mathbb{Z}^+$, only depending on $e$, such that if
\begin{itemize}
\item $k$ is an arbitrary field of characteristic zero,
\item $X$ is an at most $e$ dimensional rationally connected variety over the field $k$,
\item $G\leqq\Bir(X)$ is an arbitrary finite subgroup of the birational automorphism group of $X$,
\end{itemize}
then there exists
\begin{itemize}
\item a characteristic Abelian subgroup $A\leqq G$ of index bounded by $R$, i. e. $|G:A|<R$.
\end{itemize}
\end{thm} 

\begin{proof}
By Theorem 1.8 of \cite{PS16} we can find an Abelian subgroup of bounded index. Than by Theorem \ref{CD} we can find a characteristic Abelian subgroup of bounded index.
\end{proof}

\subsection{Bound on the minimal size of generating sets}
Now we turn our attention on finding a bound on the minimal size of generating sets of finite subgroups of the birational automorphism group of  a variety. 
It will be important for us when we investigate commutator relations (Lemma \ref{DN}), and it is crucial to have a bound on the size of a generating set of the group.\\

First we investigate the case of rationally connected varieties.
To start with, recall the theorem of Yu. Prokhorov and C. Shramov about fixed points of rationally connected varieties (Theorem 4.2 of \cite{PS14}).

\begin{thm}
\label{afp}
Let $e$ be a non-negative integer. There exists a constant $Q=Q(e)\in\mathbb{Z}^+$, only depending on $e$, such that if
\begin{itemize}
\item $X$ is rationally connected projective complex variety of dimension at most $e$,
\item $G\leqq\Aut(X)$ is an arbitrary finite subgroup of the biregular automorphism group of X,
\end{itemize}
then there exists
\begin{itemize}
\item a subgroup $H\leqq G$ such that $H$ has a fixed point in $X$, and the index of $H$ in $G$ is bounded by $Q$.
\end{itemize}
\end{thm} 

Now we can state the theorem for rationally connected varieties.

\begin{thm}
\label{mRC}
Let $e$ be a non-negative integer. There exists a constant $m=m(e)\in\mathbb{Z}^+$, only depending on $e$, such that if
\begin{itemize}
\item $k$ is an arbitrary field of characteristic zero,
\item $X$ is an at most $e$ dimensional rationally connected variety over the field $k$,
\item $G\leqq\Bir(X)$ is an arbitrary finite subgroup of the birational automorphism group,
\end{itemize}
then $G$ can be generated by $m$ elements.
\end{thm} 

\begin{proof}
Fix $k$, $X$ and $G$ with the properties described by the theorem.
Arguing as in the proof of Lemma \ref{C}, we can assume that $k$ is the field of the complex numbers.\\
Using Lemma \ref{reg}, we can assume that $X$ is smooth and projective and $G$ is a finite subgroup of the biregular automorphism group $\Aut(X)$.\\
By Theorem \ref {afp}, we can assume that $G$ has a fixed point in $X$. Denote it by $P$.\\
By Lemma 4 of \cite{Po14}, $G$ acts faithfully on the tangent space of the fixed point $P$. So $G$ can be embedded into $\GL(\T_PX)$, whence $G$ can be embedded into $\GL(e,\mathbb{C})$.
After applying Jordan's theorem on general linear groups (Theorem 2.3 in\cite{Br11}) to $G\leqq\GL(e,\mathbb{C})$, we can assume that $G$ is Abelian, hence diagonalizable. 
Finite diagonalizable subgroups of $\GL(e,\mathbb{C})$ can be generated by $e$ elements. This finishes the proof.
\end{proof}

The next theorem and its proof are essentially due to Y. Prokhorov and C. Shramov. (We use the word essentially as they only considered the case of finite Abelian subgroups (Remark 6.9 of \cite{PS14}).)

\begin{thm}
\label{m}
Let $X$ be a smooth projective complex variety. Let the pair $(Z,\phi)$ be its MRC fibration, and let $e=\dim X-\dim Z$ be the relative dimension.
There exists a constant $m=m(e,Z)\in\mathbb{Z}^+$, only depending on $e$ and on the birational class of $Z$, such that
if $G\leqq\Bir(X)$ is an arbitrary finite subgroup of the birational automorphism group of $X$, then $G$ can be generated by $m$ elements.
\end{thm}
\begin{proof}
We have already studied the rationally connected case (Theorem \ref{mRC}).
As the next step, we show the theorem in the special case when $X$ is non-uniruled. (Note that in this case the MRC fibration consists of the variety $X$ and the identity morphism). 
By Remark 6.9 of \cite{PS14}, there exists a constant $m=m(X)\in\mathbb{Z}^+$, only depending on the birational class of $X$, such that
if $A\leqq\Bir(X)$ is an arbitrary finite Abelian subgroup of the birational automorphism group, then $A$ can be generated by $m$ elements. Since $\Bir(X)$ is Jordan when $X$ is non-uniruled (Theorem \ref{nurc}),
the result on the finite Abelian groups implies the claim of the theorem in the case when $X$ is non-uniruled.\\
Now let $X$ be arbitrary.
%Arguing as in the proof of Lemma \ref{C} we can assume that $X$ is a complex variety. 
Consider the MRC fibration $\phi:X\dashrightarrow Z$.
By Lemma \ref{reg} we can assume that both $X$ and $Z$ are smooth projective varieties, and $G$ acts on them by biregular automorphisms. 
Let $\rho$ be the generic point of $Z$, and let $X_\rho$ be the generic fibre of $\phi$. $X_\rho$ is a rationally connected variety over the function field $k(Z)$, while $Z$ is a non-uniruled complex variety.\\
Let $G_\rho\leqq G$ be the maximal subgroup of $G$ acting fibrewise. $G_\rho$ has a natural faithful action on $X_\rho$, while $G/G_\rho=G_Z$ has a natural faithful action on $Z$. 
This gives a short exact sequence of finite groups
\[1\to G_\rho\to G\to G_Z\to 1.\]
By the rationally connected case there exists a constant $m_1(e)$, only depending on $e$, such that $G_\rho$ can be generated by $m_1(e)$ elements. 
By the non-uniruled case there exists a constant $m_2(Z)$, only depending on the birational class of $Z$, such that $G_Z$ can be generated by $m_2(Z)$ elements.
So $G$ can be generated by $m(e, Z)=m_1(e)+m_2(Z)$ elements. This finishes the proof.
\end{proof}

\subsection{Almost Abelian-by-Abelian group extensions}
The MRC fibration splits an arbitrary fixed variety into two parts: a non-uniruled base and a rationally connected generic fibre.
Both non-uniruled and rationally connected varieties have Jordan birational automorphism groups. 
These facts are reflected in the structure of the finite subgroups of the birational automorphism group of the fixed variety.\\
After replacing a finite subgroup $G$ of the birational automorphism group by a bounded index subgroup $G_1$, 
we can find an Abelian normal subgroup (in $G_1$) which acts fibrewise
and whose quotient is Abelian. In short, up to bounded index, a finite subgroup is an Abelian-by-Abelian group extension, where the normal subgroup acts fibrewise 
hence the action of the quotient group descends to an action on the base.\\

Before proving the main theorem of the subsection, we need to make some group theoretic preparations.

\begin{lem}
\label{Cent}
Let $R$ be positive integer. There exists a constant $C=C(R)\in \mathbb{N}$, only depending on $R$, such that if 
$$1\to N\to G \to H \to 1$$ 
is a short exact sequence of finite groups where 
\begin{itemize}
\item $N$ has at most $R$ many elements and 
\item $H$ is Abelian, 
\end{itemize}
then there exists a subgroup $G_1\leqq G$ whose index is bounded by $C$, i.e. $|G:G_1|<C$, and $G_1$ is nilpotent of class at most two.
\end{lem} 
\begin{proof}
Let $G_1$ be the centralizer group $G_1=\Cent_G(N)=\{g\in G|\; ng=gn\;\forall n\in N\}$.  $G$ acts on $N$ by conjugation, and the kernel of this action is  $G_1$.
Therefore $G/G_1$ embeds into the automorphism group of $N$ which has cardinality at most $R!$. Hence $|G:G_1|\leqq R!$.
It is clear by construction that $G_1$ is the central extension of the Abelian groups $N\cap G_1$ and $\Imag(G_1\to H)$.
Hence $G_1$ is nilpotent of class at most two (Proposition \ref{CE}).
\end{proof}

\begin{lem}
\label{DN}
Let $R$ and $m$ be positive integers. There exists a constant $D=D(R,m)\in \mathbb{N}$, only depending on $R$ and $m$, such that if
\begin{itemize} 
\item $G$ is a nilpotent group of class at most two 
\item $G$ can be generated by $m$ elements,  
\item the cardinality of the commutator subgroup $G'=[G,G]$ is at most $R$, i.e. $|G'|<R$,
\end{itemize}
then $G$ has an Abelian subgroup $A\leqq G$ whose index is bounded by $D$. 
\end{lem} 
\begin{proof}
Fix a generating system $g_1,...,g_m\in G$. Consider the group homomorphisms 
\begin{gather*}
\varphi_{i}:G\to G'\\
g\mapsto [g,g_i],
\end{gather*} 
where $1\leqq i\leqq m$ (Proposition \ref{ICmap}), i.e. for every generator we assign a group homomorphism using the commutator.
Let $A$ be the intersection of the kernels.
$$A=\bigcap\limits_{1\leqq i \leqq m} \Ker\varphi_{i}$$ 
Using the fact that for a nilpotent group of class at most two the commutators give group homomorphisms in both variables if we fix the other variable (Proposition \ref{ICmap}), 
one can show that all commutators of $A$ vanish. Hence $A$ is Abelian.\\
On the other hand $A$ is the intersection of $m$ subgroups of index at most $|G'|\leqq R$. Hence the index of $A$ is bounded in terms of $R$ and $m$. This finishes the proof.
\end{proof}

\begin{lem}
\label{DN2}
Let $R$ and $m$ be positive integers. There exists a constant $E=E(R,m)\in \mathbb{N}$, only depending on $R$ and $m$, such that if 
$$1\to N\to G \to H \to 1$$ 
is a short exact sequence of finite groups where 
\begin{itemize}
\item $N$ has at most $R$ many elements, 
\item $H$ is Abelian, and
\item every subgroup of $G$ can be generated by at most $m$ elements, 
\end{itemize}
then $G$ has an Abelian subgroup $G_1\leqq G$ whose index is bounded by $E$. 
\end{lem} 
\begin{proof}
By Lemma \ref{Cent}, after replacing $G$ with $\Cent_G(N)$,  $N$ with $N\cap \Cent_{G}(N)$ and $H$  with the image group $\Imag(\Cent_G(N)\to H)$,
we can assume that  $G$ is the central extension of the Abelian groups $N$ and $H$. Therefore $G$ is nilpotent of class at most two, and we can apply the previous lemma.
Indeed $G/N$ is Abelian, hence $G'\leqq N $, therefore $|G'|\leqq|N|\leqq R$. This finishes the proof. 
\end{proof}

Now we are ready to prove the main theorem of the subsection.

\begin{thm}
\label{AbyA}
Let $X$ be a smooth projective complex variety. Let the pair $(Z,\phi)$ be the MRC fibration of $X$ and $e=\dim X-\dim Z$ be the relative dimension.
There exists a constant $E=E(e, Z)\in\mathbb{Z}^+$, only depending on $e$ and the birational class of $Z$, such that if
\begin{itemize}
\item $G\leqq\Bir(X)$ is an arbitrary finite subgroup of the birational automorphism group of $X$,
\end{itemize}
then there exists
\begin{itemize}
\item a subgroup $G_1\leqq G$ with index bounded by $E$, i.e. $|G:G_1|<E$, such that
\item $G_1$ contains an Abelian normal subgroup $N_1\unlhd G_1$, such that
\item $N_1$ acts trivially on the non-uniruled base $Z$; furthermore 
\item the quotient group $G_1/N_1$ is Abelian.
\end{itemize}
\end{thm}

\begin{proof}
By Lemma \ref{reg} we can assume that both $X$ and $Z$ are smooth projective varieties, and $G$ acts on them by biregular automorphisms. 
Let $\rho$ be the generic point of $Z$, and let $X_\rho$ be the generic fibre of $\phi$. $X_\rho$ is a rationally connected variety over the function field $k(Z)$, while $Z$ is a non-uniruled complex variety.\\
Let $G_\rho\leqq G$ be the maximal subgroup of $G$ acting fibrewise. $G_\rho$ has a natural faithful action on $X_\rho$, while $G/G_\rho=G_Z$ has a natural faithful action on $Z$. 
This gives a short exact sequence of finite groups 
\[1\to G_\rho\to G\to G_Z\to 1.\]
By Theorem \ref{nurc} there exists a constant $S=S(Z)$, only depending on the birational class of $Z$, such that $G_Z$ contains an Abelian subgroup $A_Z$ of index bounded by $S$. After replacing $G_Z$ by $A_Z$,
$G$ by the preimage of $A_Z$, and $G_{\rho}$ by the intersection of the preimage of $A_Z$ and $G_{\rho}$, we can assume that $G_Z$ is Abelian.\\ 
By Theorem \ref{RCJor} there exists a constants $R=R(e)$, only depending on $e$, such that $G_{\rho}$ contains a characteristic Abelian subgroup $N\unlhd G_{\rho}$ of index bounded by $R$.
Since $N$ is characteristic in $G_{\rho}$, and $G_{\rho}$ is normal in $G$, $N$ is a normal subgroup of $G$. Consider the short exact sequence
 \[1\to \overline{G_\rho}\to \overline{G}\to G_Z\to 1,\]
where $\overline{G_\rho}=G_{\rho}/N$ and $\overline{G}=G/N$. The cardinality of $\overline{G_{\rho}}$ is bounded by $R$, while $G_Z$ is Abelian. 
Moreover Theorem \ref{m} show us that there exists a constant $m=m(e, Z)$, only depending on $e$ and the birational class of $Z$, such that every finite subgroup of $\overline{G}$ can be generated by $m$ elements.
Therefore, by Lemma \ref{DN2}, there exists a constant $E=E(R,m)=E(e, Z)$, only depending on $e$ and the birational class of $Z$, such that $\overline{G}$ contains an Abelian subgroup $\overline{H}$ of index bounded by $E$.
Let $G_1$ be the preimage of $\overline{H}$ in $G$.\\
The data $E_1=SE, N_1=N\cap G_1, G_1$ satisfies the claim of the lemma. This finishes the proof.
\end{proof}

\section{Building a vector bundle from the MRC fibration}
\label{MRCtoVB}
The core result of this section is Lemma \ref{YW}. Given the data $X, Z, E_X\to X,f:X\to Z, G, N$ it returns the data $Y, W, E_Y\to Y, g:Y\to W, G_0, N_0$.
It starts with a vector bundle $E_X$ over $X$ endowed with a faithful action of a finite group $G$, 
a $G$-equivariant dominant morphism $f:X\to Z$ of smooth projective complex varieties, such that $f$ has rationally connected general fibres,
and an Abelian normal subgroup $N\unlhd G$ which acts trivially on $Z$ and non-trivially on $X$. The lemma results a similar kind of data,
however it lowers the dimension of the base variety of the vector bundle: $\dim X>\dim Y$. 
Moreover $\dim E_X=\dim E_Y$ (i.e. the dimensions of the total spaces of the vector bundles are equal), $G_0\leqq G$ is a subgroup of
bounded index, $N_0=G_0\cap N$, and the lemma allows the $N_0$-action on $Y$ to be trivial.\\
We will apply this result repeatedly with initial data $X, Z, X\cong X, \phi: X\to Z, G, N$, 
where $(Z, \phi)$ is the MRC fibration of $X$ (after performing a smooth regularization, Lemma \ref{reg}) and 
$X\cong X$ is the vector bundle with zero dimensional linear spaces as fibres. 
We will end up with a $G_0$-equivariant vector bundle $E_Y\to Y$, where $G_0\leqq G$ is a (finite) subgroup of bounded index. 
Furthermore, $G_0$ acts faithfully via vector bundle automorphisms on $E_Y$, and the $G_0$-action on $Y$ descends to the action of $G_0/N_0$,
where $N_0=G_0\cap N$ is the corresponding  normal subgroup (Lemma \ref{MainLemma}).\\
Additionally, after replacing $G$ and $N$ by $G_1$ and $N_1$ from Theorem \ref{AbyA} in the initial data, 
we can assume that the $G_0/N_0$-action is Abelian (Theorem \ref{VB}).\\  

The proof of Lemma \ref {YW} consist of four major steps. First, by the help of the $G$-equivariant relative MMP on $X$ over $Z$ we find a Mori fibration $F\to B$, 
such that $N$ acts on $F$ non-trivially, and $N$ acts on $B$ trivially. 
Moreover, after replacing $X$ by another smooth birational model, we have a dominant morphism $X\to F$ with rationally connected general fibres (Lemma \ref{F}).\\
Secondly, by the help of boundedness of Fano varieties and Jordan's theorem on linear groups, we find a bounded index subgroup $G_0\leqq G$, 
and a proper $G_0$-invariant closed subvariety $W\subsetneqq F$ on which $N_0=N\cap G_0$ acts trivially \linebreak (Lemma \ref{W}).\\
As the third step, we use Lemma \ref{PullBack1} (which describes a technique to build rationally connected fibrations) 
to find a proper $G_0$-invariant closed subvariety $Y\subsetneqq X$ and a $G_0$-equivariant dominant morphism $Y\to W$ with rationally rationally connected general fibres (Lemma \ref{PullBackF}).\\
Finally, we use the normal bundle $\nu (Y\subsetneqq E_X)$ and resolution of singularities and indeterminacies to finishes the proof of Lemma \ref{YW}.\\

We start with showing the existence of the required Mori fibration $F\to B$. Our main tool for this is the MMP.

\begin{lem}
\label{F}
Let
\begin{itemize}
\item $G$ be a finite group,
\item $X$ and $Z$ be smooth complex varieties endowed with (not necessarily faithful) $G$-actions via biregular automorphisms,
\item $N\unlhd G$ be a normal Abelian subgroup such that $N$ acts trivially on $Z$, i.e. the $G$-action on $Z$ descends to a $G/N$-action,
\item $f:X\to Z$ be a $G$-equivariant dominant morphism with rationally connected general fibres,
\end{itemize}
then, either 
\begin{itemize}
\item $N$ acts trivially on $X$, i.e. the $G$-action on $X$ descends to a $G/N$-action, 
\end{itemize} 
or there exists 
\begin{itemize}
\item a smooth projective complex variety $X_0$, such that
\item $X_0$ is endowed with a (not necessarily faithful) $G$-action via biregular automorphisms; moreover there exists
\item a $G$-equivariant dominant birational morphism $X_0\to X$; furthermore there exists
\item a normal projective complex variety $F$, such that
\item $F$ is endowed with a (not necessarily faithful) $G$-action via biregular automorphisms, such that $N\leqq G$ acts non-trivially on $F$,
\item $F$ is $G\mathbb{Q}$-factorial and has terminal singularities; moreover there exist
\item a dominant $G$-equivariant morphism $f_0:X_0\to F$ with rationally connected general fibres,
\item a normal projective complex variety $B$, such that
\item $B$ is endowed with a (not necessarily faithful) $G$-action via biregular automorphisms, such that $N\leqq G$ acts trivially on $B$,  i.e. the $G$-action on $B$ descends to a $G/N$-action,
\item $\dim B<\dim F$; moreover there exists
\item a dominant $G$-equivariant morphism $h:F\to B$ whose general fibres are Fano varieties (with terminal singularities).
\end{itemize}
\end{lem}

\begin{proof}
Fix $Z$, $G$ and $N$. We apply induction on the relative dimension $e=\dim X-\dim Z$. If $e=0$ then the general fibres of $f:X\to Z$ are zero dimensional rationally connected varieties, hence they are rational points. 
This implies that $X$ is birational to $Z$. Therefore $N$ acts trivially on $X$. This finishes the proof of the case when $e=0$. So let $e>0$.\\
If $N$ acts trivially on $X$, then the proof is finished, so assume otherwise.
Since the general fibres of $f$ are rationally connected varieties and $X$ is smooth,
we can run a $G$-equivariant MMP on $X$ over $Z$ (Theorem \ref{MMP}). It stops at a $G$-equivariant Mori fibre space over $Z$. Therefore we have a commutative $G$-equivariant diagram
 \[
\xymatrix{
X\ar@ {-->} [r]^{\sim} \ar[rd] & F_1 \ar[r] \ar[d] & B_1 \ar[ld]\\
& Z,
}
\]
where  $G$ acts via biregular automorphisms on the varieties, $F_1$ is a normal, $G\mathbb{Q}$-factorial projective variety with terminal singularities, the rational map $X\dashrightarrow F_1$ is a birational automorphism,
$B_1$ is a normal projective variety such that $\dim B_1<\dim F_1=\dim X$, and the general fibres of the dominant morphism $F_1\to B_1$ are Fano varieties with terminal singularities.
Moreover the dominant morphism $B_1\to Z$ has rationally connected general fibres.\\
If $N$ acts trivially on $B_1$, then consider a smooth $G$-equivariant resolution of indeterminacies of $X\dashrightarrow F_1$
 \[
\xymatrix{
X_0 \ar[r]^{\sim} \ar[rd] & X \ar@ {-->}[d]\\
& F_1,
}
\] 
here $X_0$ is a smooth projective complex variety endowed with a $G$-action, furthermore $X_0\to X$ is a $G$-equivariant dominant birational morphism. 
Hence the data $X_0,X_0\overset{\sim}{\to} X, F_1, B_1, X_0\to F_1\to B_1$ satisfies the conditions of the lemma. 
This finishes the proof when $N$ acts trivially on $B_1$.\\
Otherwise let $C_1\to B_1$ be a $G$-equivariant resolution of singularities. We can apply induction to the data $G, C_1, Z, N, C_1\to Z$, since the relative dimension drops, 
i.e. $\dim C_1-\dim Z=\dim B_1-\dim Z<\dim X-\dim Z$.
The induction results data $D_1, D_1\overset{\sim}{\to}C_1, F_2, B_2,  D_1\to F_2\to B_2$ with the properties satisfying the lemma (and suggested by the notion). 
Once again we can finish the proof by $G$-equivariantly and smoothly  resolving the indeterminacies of $X\dashrightarrow D_1$. 
(Note that, by construction, the dominant morphism $X_0\to F_2$ has rationally connected general fibres.)
This finishes the proof.
\end{proof}

We continue with the lemma about the proper closed subvariety  $W\subsetneqq F$ (formed by the fixed points of $N_0\leqq N$). Here we use the result of C. Birkar on boundedness of Fano varieties.

\begin{lem}
\label{W}
Let $e$ be a positive integer. There exists a constant $C=C(e)\in\mathbb{Z}^+$, only depending on $e$, such that, if 
\begin{itemize} 
\item $G$ is a finite group,
\item$F$ and $B$ are normal projective complex varieties endowed with (not necessarily faithful) $G$-actions via biregular automorphisms, 
\item $e\geqq\dim F-\dim B>0$ is the relative dimension, 
\item $h:F\to B$ is a $G$-equivariant dominant morphism whose general fibres are Fano varieties with terminal singularities,
\item $N\unlhd G$ is an Abelian normal subgroup such that $N$ acts non-trivially on $F$ and trivially on $B$,
\end{itemize}
then there exist
\begin{itemize} 
\item a subgroup $G_0\leqq G$ such that its index is bounded by $C$, i.e. $|G:G_0|<C$,
\item a closed $G_0$-invariant subvariety $W\subsetneqq F$ of $X$ such that 
\item $\dim W<\dim F$, and
\item $N_0=G_0\cap N$ acts trivially on $W$,
hence the $G_0$-action on $W$ descends to a $G_0/N_0$-action.
\end{itemize}
\end{lem}

\begin{proof}
Let $\rho$ be the generic point of $B$, and let $F_\rho$ be the generic fibre of $h$. $F_\rho$ is an at most $e$ dimensional Fano variety with terminal singularities over the function field $k(B)$.\\
Consider the generic fibre $F_{\rho}$.
By boundedness of Fano varieties  (see Corollary \ref{Fam2}), we can find constants $E=E(e), m=m(e),N=N(e)\in\mathbb{Z}^+$, only depending on $e$, 
such that 
%for any at most $e$ dimensional Fano variety (over an arbitrary field of characteristic zero) with terminal singularities,
%the $m$-th power of the anticanonical divisor embeds the variety into  an at most $(N-1)$ dimensional projective space with degree at most $E$. In particular 
$-mK_{F_{\rho}}$ defines a closed embedding 
for which the closed subvariety $F_{\rho}\subseteqq\mathbb{P}(\CoH^0(F_{\rho},-mK_{F_{\rho}})^{\star})$ has degree at most $E$ and $\dim_{k(B)}\CoH^0(F_{\rho},-mK_{F_{\rho}})\leqq N$.\\
Moreover, the biregular automorphism group $\Aut_{F_{\rho}}$ has a natural faithful linear action on the vector space $\CoH^0(F_{\rho},-mK_{F_{\rho}})^{\star}$, 
which makes the closed embedding  $F_{\rho}\hookrightarrow\mathbb{P}(\CoH^0(F_{\rho},-mK_{F_{\rho}})^{\star})$ $\Aut_{F_{\rho}}$-equivariant.\\  
Let $A$ be the image group $\Imag(N\to\Aut F{_\rho})$. By the above observation, $A$ is an Abelian group which acts faithfully via linear automorphisms 
on the closed subvariety $F_{\rho}\subseteqq\mathbb{P}(\CoH^0(F_{\rho},-mK_{F_{\rho}})^{\star})$.\\
Let $V\subsetneqq F_{\rho}$ be the reduced closed subscheme formed by the fixed points of $A$. (Note that $\dim V<\dim F$ since $A$ acts non-trivially on $F$.) 
As $A$ acts linearly on the closed subvariety $F_{\rho}\subseteqq\mathbb{P}(\CoH^0(F_{\rho},-mK_{F_{\rho}})^{\star})$, 
$V$ is the intersection of $F_{\rho}$ and the union of at most $N$ many linear spaces. Since the degree of $F$ is at most $E$, $V$ has at most $NE$ many components. Let $C=(NE)!$ (clearly it only depends on $e$) and 
let $V_1$ be an arbitrary irreducible component of $V$ endowed with the reduced subscheme structure.\\
By the definition of $A$, the points of $F_{\rho}$ fixed by $A$ are the same as the points of $F_{\rho}$ fixed by $N$. Hence $V$ is the closed subscheme formed by the fixed points of $N$. 
As $N$ is a normal subgroup of $G$, $V$ is $G$-invariant. Hence $G$ induces an action on the components of $V$. 
Therefore there exists a subgroup $G_0\leqq G$ with index bounded by $C=(NE)!$, which fixes all components of $V$. In particular $V_1$ is $G_0$-invariant.
Let $N_0=N\cap G_0$, clearly $V_1$ is formed by fixed points of $N_0$.\\
Let $W$ be the closure of $V_1$ in the complex variety $F$ endowed with the reduced scheme structure. Clearly $\dim W<\dim F$, and $W$ is a $G_0$-invariant closed subvariety of $F$ formed by fixed points of $N_0$.
Hence the $G_0$-action on $W$ descends to a $G_0/N_0$-action. This finishes the proof of the lemma.  
\end{proof}

Now we construct the proper closed subvariety $Y\subsetneqq X$ and the dominant $G$-equivariant morphism $Y\to W$ with rationally connected fibres.

\begin{lem}
\label{PullBackF}
Let 
\begin{itemize} 
\item $G_0$ be a finite group,
\item $X$ and $F$ be normal projective complex varieties endowed with (not necessarily faithful) $G_0$-actions via biregular automorphisms, such that
\item $F$ is $G_0\mathbb{Q}$-factorial and has terminal singularities, moreover 
\item let $f:X\to F$ be a $G_0$-equivariant dominant morphism with rationally connected fibres,
\item $W\subsetneqq F$ be a $G_0$-invariant closed subvariety, such that $\dim W<\dim F$,
\end{itemize}
then there exists
\begin{itemize} 
\item a closed subvariety $Y\subsetneqq X$, such that
\item  $\dim Y<\dim X$,
\item $Y$ is $G_0$-invariant,
\item$f$ restricts to a $G_0$-equivariant dominant morphism from $Y$ to $W$  (i.e. $f|_{Y}:Y\to W$ is dominant) with rationally connected general fibres.  
 \[
\xymatrix{
Y\ar@{^{(}->}[r] \ar[d]^{f|_{Y}} & X \ar[d]^f\\
W\ar@{^{(}->}[r] \ & F
}
\]
\end{itemize}
\end{lem}

The main technical tool of this section is a lemma of Yu. Prokhorov and C. Shramov (Lemma 3.4 of \cite{PS16}).
It is a $G$-equivariant version of the result of C. D. Hacon and J. McKernan (Corollary 1.9 of \cite{HM07}). 
It gives us a technique to find rationally connected fibrations. We state the lemma below.

\begin{lem}
\label{PullBack1}
Let 
\begin{itemize} 
\item $G$ be a finite group,
\item $X_1$ and $X_2$ be complex normal quasi-projective varieties endowed with (not necessarily faithfully) $G$-actions via biregular automorphisms,
\item $X_1$ have Kawamata log terminal singularities,
\item $f:X_1\to X_2$ be a $G$-equivariant dominant projective morphism with connected fibres, such that 
\item the anticanonical bundle of $X_1$ is $f$-ample, moreover
\item let $Y_2\subsetneqq X_2$ be a $G$-invariant closed subvariety, such that $\dim Y_2<\dim X_2$. 
\end{itemize} 
Then there exists a closed subvariety $Y_1\subsetneqq X_1$  such that 
\begin{itemize} 
\item $\dim Y_1<\dim X_1$, 
\item $Y_1$ is $G$-invariant,
\item $f$ restricts to a $G$-equivariant dominant morphism from $Y_1$ to $Y_2$, i.e. $f|_{Y_1}:Y_1\to Y_2$ is dominant,
\item $f|_{Y_1}$ has rationally connected general fibres.  
\end{itemize} 
\end{lem}

 We will also need a technical result about composition of dominant morphisms with rationally connected fibres. 
 It is a simple consequence of Corollary 1.3 in \cite{GHS03}.

\begin{lem}
\label{RC}
Let $X, Y$ and $Z$ be complex projective varieties. Let $f:X\to Y$ and $g:Y\to Z$ be dominant morphisms between them with rationally connected general fibres.
Then the general fibres of $g\circ f:X\to Z$ are rationally connected. 
\end{lem}

\begin{lem}
\label{PullBack2}
Let 
\begin{itemize} 
\item $G$ be a finite group,
\item $X_1$ and $X_2$ be normal projective complex varieties endowed with (not necessarily faithful) $G$-actions via biregular automorphisms, 
\item $X_1$ have terminal singularities, 
\item $F$  be a normal projective complex variety endowed with a (not necessarily faithful) $G$-action via biregular automorphisms,
\item $h_1:X_1\to F$ and $h_2:X_2\to F$ be $G$-equivariant dominant morphisms,
\item $f:X_1\dashrightarrow X_2$ be a $G$-equivariant rational map such that the following $G$-equivariant commutative diagram
 \[
\xymatrix{
X_1\ar@ {-->}[r]^f \ar[d]_{h_1} & X_2 \ar[ld]^{h_2}\\
F
}
\]
is a step of a $G$-equivariant relative MMP over the base variety $F$, 
i. e. $f$ is either a Mori fibration, a flip or a divisorial contraction, 
\item $W$ be a $G$-invariant closed subvariety of $F$,
\item $Y_2\subsetneqq X_2$ be a $G$-invariant closed subvariety, such that $\dim Y_2<\dim X_2$, 
\item moreover assume that,  $h_2$ restricts to a dominant $G$-equivariant morphism from $Y_2$ to $W$ ( i.e. $h_2|_{Y_2}:Y_2\to W$ is dominant)
such that the general fibres of $h_2|_{Y_2}$ are rationally connected.  
\end{itemize} 
Then there exists a closed subvariety $Y_1\subsetneqq X_1$  such that 
\begin{itemize} 
\item $\dim Y_1<\dim X_1$, 
\item $Y_1$ is $G$-invariant,
\item $h_1$ restricts to a dominant $G$-equivariant morphism from $Y_1$ to $W$  (i.e. $h_1|_{Y_1}:Y_1\to W$ is dominant) with rationally connected general fibres.  
\end{itemize} 
\end{lem} 

\begin{proof}
Assume first that $f$ is either a Mori fibration or a divisorial contraction. 
Then we can apply Lemma \ref{PullBack1} to get a $G$-invariant closed subvariety $Y_1\subsetneqq X_1$ which dominates $Y_2$, and for which the fibres of $f|_{Y_1}:Y_1\to Y_2$ are rationally connected.
Hence by Lemma \ref{RC} the dominant morphism $h_1|_{Y_1}:Y_1\to W$ has rationally connected general fibres. This finishes the proof of the above cases.\\
If $f$ is a flip then there is a normal projective variety $X_3$ endowed with a $G$-action over the base variety $F$, and we have a $G$-equivariant commutative diagram 
 \[
\xymatrix{
X_1\ar@ {-->}[r]^f \ar[d]_{k_1} & X_2 \ar[ld]_{k_2} \ar[ldd]^{h_2}\\
X_3 \ar[d]_{h_3}\\
F,
}
\]
where $k_1:X_1\to X_3$ is a dominant $G$-equivariant (projective) morphism with connected fibres, such that the anticanonical bundle of $X_1$ is $k_1$-ample. Moreover $h_1=h_3\circ k_1$.\\
Let $Y_3\subsetneqq X_3$ be the scheme theoretic image of $Y_2$ by the projective morphism $k_2$. (Note that $\dim Y_3<\dim X_3=\dim X_2$.) 
Clearly $Y_2$ is a $G$-invariant closed subvariety, it dominates $W$, and the general fibres of $h_3|_{Y_3}: Y_3\to W$ are rationally connected.\\
Using the $G$-invariant closed subvariety $Y_3\subsetneqq X_3$, we can apply Lemma \ref{PullBack1} to get a $G$-invariant closed subvariety $Y_1\subsetneqq X_1$ 
which dominates $Y_2$, moreover for which the general fibres of $k_1|_{Y_1}:Y_1\to Y_3$ are rationally connected. Hence by Lemma \ref{RC}, the general fibres of the dominant morphism $h_1|_{Y_1}:Y_1\to W$ are also rationally connected.
This finishes the proof of the lemma.
\end{proof}

Now we are ready to prove Lemma \ref{PullBackF}.

\begin{proof}[Proof of Lemma \ref{PullBackF}]
First realize that it is enough to prove the lemma under the extra assumption that $X$ is smooth. Indeed, if $X$ is not smooth, then we can $G_0$-equivariantly resolve the singularities of $X$. Let it be $X_1\to X$.
By the smooth case, there exists $Y_1\subsetneqq X_1$ satisfying the conditions of the lemma for $X_1$. Its scheme theoretic image $Y\subsetneqq X$ will satisfy the conditions of the lemma for $X$.\\
We apply induction on the relative dimension $e=\dim X-\dim F$. Assume that $e=0$. By the observation above we can assume that $X$ is smooth. 
Since the general fibres of $f$ are zero dimensional rationally connected varieties, i.e. rational points, $f$ is a birational morphism.
If $f$ is an isomorphism then the case is finished, so assume otherwise. We can run a $G$-equivariant MMP on $X$ over $Z$ by Theorem \ref{MMP2}.
The MMP factorizes $f$ into a sequence of $G$-equivariant flips and divisorial contractions. Hence the last map of the MMP is of the form of
 \[
\xymatrix{
X_1\ar@ {-->}[r]^m\ar[d] & X_2 \ar[ld]^{\cong}\\
F,
}
\]
where $m$ is either a flip or a divisorial contraction, and $X_2\to F$ is an isomorphism. 
We can use Lemma \ref{PullBack2} and apply induction on the number of steps of the MMP.
The induction starts as the MMP stops at a variety isomorphic to $F$. So if the MMP has one step, then we can apply Lemma \ref{PullBack2} directly. 
This finishes the proof when $e=0$.\\
So let $X$ be a smooth projective variety and assume that $e>0$. We can run a $G_0$-equivariant MMP on (the smooth projective) variety $X$ relative to $F$ (by Theorem \ref{MMP}). It ends up at a Mori fibre space $m:X_1\to B$.
 \[
\xymatrix{
X\ar@ {-->}[r]^{\sim}\ar[rd] & X_1\ar[r]^m\ar[d] & B \ar[ld]\\
& F
}
\]
As $B$ is a normal projective complex variety endowed with a $G_0$-action such that $\dim B<\dim X$, we can apply induction to $B$. (Note that $B$ is not necessarily smooth.) 
Therefore we can find a closed subvariety $V\subsetneqq B$ which satisfies the conditions of the lemma. By Lemma \ref{PullBack2},
we can apply induction on the number of steps of the MMP to pull back $V\subsetneqq B$ to $Y\subsetneqq X$. This finishes the proof.
\end{proof}

We need to consider a couple of lemmas about $G$-equivariant vector bundles before proving Lemma \ref{YW}.

\begin{lem}
\label{NB}
Let $X$ be a smooth complex variety, and let $G\leqq \Aut (X)$ be a subgroup of its automorphism group. Let $Y\subseteqq X$ be a $G$-invariant closed subvariety of $X$.
Then the induced $G$-action on the normal bundle of $Y$, denoted by $\nu(Y\subseteqq X)$, is faithful.
\end{lem}

\begin{proof}
The $G$-action on $Y\subseteqq X$ induces a $G$-action on the normal bundle, therefore we only need to prove that it is faithful.\\
Let $g\in G$ be an arbitrary group element, we will prove that the action of $g$ on $\nu(Y\subseteqq X)$ cannot be trivial. If the action of $g$ on $Y$ is not trivial, then we have nothing to prove, so assume otherwise.\\
Let $P\in Y\subseteqq X$ be an arbitrary closed point of $Y$. Since the action of $g$ on $X$ is not trivial, however it fixes the point $P$, $g$ should act non-trivially on the tangent space $\T_PX$ (Lemma 4 of \cite{Po14}). 
Since $g$ acts trivially on the closed subvariety $Y$, the action of $g$ is trivial on $\T_PY$.  Putting these together, we can see that 
the action of $g$ on the normal space at the point $P$ $\nu_P(Y\subseteqq X)=\T_PX/\T_PY$ is non-trivial. Since $g\in G$ is an arbitrary element, this finishes the proof.
\end{proof}

\begin{defn}
Let $X$ be a variety, $E_X\to X$ be a vector bundle over $X$ and $G$ be a group. Assume that $G$ acts on $X$ (not necessarily faithfully) via biregular automorphisms.
We say that $G$ acts on $E_X\to X$ via vector bundle automorphisms if the commutative diagram below is $G$-equivariant, 
i.e. the $G$-action on the total space $E_X$ is compatible with the predescribed action on the base variety $X$,
\[
\xymatrix{
E_X\ar[r]\ar[d] & E_X \ar[d]\\
X\ar[r] & X, 
}
\]
and $G$ respects the vector bundle structure.\\
By $\dim E_X$ we denote the dimension of the total space of the vector bundle,
i.e. $\dim E_X=\rk E_X+\dim X$ (where $\rk E_X$ is the rank of the vector bundle). 

\end{defn}

\begin{lem}
\label{PBVB}
Let
\begin{itemize}
\item $G$ be a group,
\item $X_1$ and $X$ be projective varieties endowed with (not necessarily faithful) $G$-actions via biregular automorphisms, 
\item $f:X_1\to X$ be a $G$-equivariant dominant morphism,
\item $E_X\to X$ be a vector bundle, on which $G$ acts faithfully via vector bundle automorphisms, 
\item $E_{X_1}\to X_1$ be the vector bundle obtained by pulling back the vector bundle $E_X\to X$ along $f:X_1\to X$.
 \[
\xymatrix{
E_{X _1}\ar[r]\ar[d] & E_X \ar[d]\\
X_1\ar[r] & X 
}
\]
\end{itemize}
Then $G$ acts faithfully on $E_{X_1}\to X_1$ via vector bundle automorphisms.
\end{lem}

\begin{proof}
Clearly $G$ acts on $E_{X_1}\to X_1$ $G$-equivariantly, via vector bundle automorphisms. $E_{X_1}\to E_X$ is surjective, as so does $X_1\to X$. 
Hence the $G$-action on $E_{X_1}$ is faithful, as it is faithful on $E_X$.
\end{proof}

As the next step, we finally prove the core result of the section, Lemma \ref{YW}.

\begin{lem}
\label{YW}
Let $d$ be a non-negative integer. There exists a constant $C=C(d)\in\mathbb{Z}^+$, only depending on $d$, such that, if 
\begin{itemize} 
\item $G$ is a finite group,
\item$X$ and $Z$ are smooth projective complex varieties endowed with (not necessarily faithfully) $G$-actions via biregular automorphisms, 
\item $d=\dim X$ is the dimension of $X$, and $\dim X\geqq\dim Z$,
\item $f:X\to Z$ is a $G$-equivariant dominant morphism with rationally connected general fibres,
\item $N\unlhd G$ is a normal Abelian subgroup such that $N$ acts trivially on $Z$,
i.e. the $G$-action on $Z$ descends to a $G/N$-action on $Z$,
\item $E_X\to X$ is a vector bundle, on which $G$ acts faithfully via vector bundle automorphisms (in such a way that the $G$-action is compatible with the $G$-action on $X$), 
%\item $e=\dim E_X$ is the dimension of the vector bundle,
\end{itemize} 
then either
\begin{itemize}
\item $N$ acts trivially on $X$,
i.e. the $G$-action on $X$ descends to a $G/N$-action on $X$,
\end{itemize} 
or there exists
\begin{itemize} 
\item a subgroup $G_0\leqq G$ such that its index is bounded by $C$, i.e. $|G:G_0|<C$, moreover there exist
\item$Y$ and $W$ smooth projective complex varieties endowed with (not necessarily faithful) $G_0$-actions via biregular automorphisms, such that
\item $\dim Y<d$,
\item $N_0=G_0\cap N$ acts trivially on $W$, i.e.
the $G_0$-action on $W$ descends to a $G_0/N_0$-action, moreover, there exist
\item a $G_0$-equivariant dominant morphism $g:Y\to W$ with rationally connected general fibres,
\item a vector bundle $E_Y\to Y$, on which $G_0$ acts faithfully via vector bundle automorphisms (in such a way that the $G_0$-action is compatible with the $G_0$-action on $Y$),  moreover
\item $\dim E_Y$=$\dim E_X$, i.e. the dimension of the total space of the vector bundle over $Y$ 
is equal to the dimension of the total space of the vector bundle over $X$.
%the dimension of the vector bundle $E_Y$ is equal to $e=\dim E_X$.
\end{itemize}
\end{lem}

\begin{proof}
We can assume that $N$ acts non-trivially on $X$.
Fix $d, G, X, Z, f, N, E_X$ as in the lemma. 
By Lemma \ref{F} we can find data $X_0,X_0\overset{\sim}{\to}X, F, B, X_0\to F\to B$. 
By Lemma \ref{PBVB} the pullback of the vector bundle $E_X\to X$ along  $X_0\to X$ results a vector bundle $E_{X_0}\to X_0$ on which $G$ acts faithfully via vector bundle automorphisms. 
Replace $X$ by $X_0$ and $E_X$ by $E_{X_0}$.  By Lemma \ref{W} we can find data $C, G_0, N_0, W$. 
As the next step, by Lemma \ref{PullBackF} we can pull back $W$ to a proper closed $G_0$-invariant subvariety $Y\subsetneqq X$ such that $Y\to W$ is $G_0$-equivariant, dominant and has rationally connected general fibres.
Therefore we have the following $G_0$-equivariant diagram. 
 \[
\xymatrix{
                                                        & E_X\ar[d]\\
Y\ar@{^{(}->}[r] \ar[d]^{f|_{Y}} & X \ar[d]^f\\
W\ar@{^{(}->}[r]                          & F
}
\]
We left with the task to replace $Y$ and $W$ by smooth varieties, 
moreover we need to establish a vector bundle over the smooth replacement of $Y$ on which $G_0$ acts faithfully via vector bundle automorphisms.\\
First resolve $G_0$-equivariantly the singularities of the closed subvariety $Y\subsetneqq X$. This results a $G_0$-equivariant commutative diagram
 \[
\xymatrix{
Y_1\ar@{^{(}->}[r]\ar[d]^{\sim} & X_1 \ar[d]^{\sim}\\
Y\ar@{^{(}->}[r] & X, 
}
\]
where $Y_1$ and $X_1$ are smooth varieties endowed with $G_0$-actions via biregular automorphisms,
and the $G_0$-equivariant morphisms $Y_1\to Y$ and $X_1\to X$ are dominant and birational.\\
Let $E_{X_1}\to X_1$ be the vector bundle over $X_1$ obtained by pulling back the vector bundle $E_X\to X$. 
By Lemma \ref{PBVB} $E_{X _1}\to X_1$ is a vector bundle on which $G_0$ acts faithfully via vector bundle automorphisms.
Consider the normal bundle $\nu(Y_1\subseteqq E_{X_1})$.  
By Lemma \ref{NB} $G_0$ acts on $\nu(Y_1\subseteqq E_{X_1})$ faithfully via vector bundle automorphisms. Clearly  $\dim \nu(Y_1\subseteqq E_{X_1})=\dim E_{X_1}=\dim E_X$.\\
Let $W_2\to W$ be a $G_0$-equivariant resolution of the singularities of $W$. We can $G_0$-equivariantly (and smoothly) resolve the indeterminacies of the dominant rational map $Y_1\dashrightarrow W_2$.
This gives us the following $G_0$-equivariant commutative diagram.
 \[
\xymatrix{
 & Y_2 \ar[d] \ar[ldd]\\
 & Y_1 \ar[d] \ar@ {-->}[ld]\\
W_2\ar[r] & W
}
\]
$Y_2\to W_2$ is a $G_0$-equivariant morphism of smooth projective varieties. Moreover $Y_2$ is $G_0$-equivariantly birational to $Y_1$, hence it is $G_0$-equivariantly birational to $Y$. In particular $\dim Y_2<d$.
Also note that the general fibres of $Y_2\to W_2$ are rationally connected varieties.\\
Let $E_{Y_2}\to Y_2$ be the pullback of the vector bundle $\nu(Y_1\subseteqq E_{X_1})\to Y_1$ along the morphism $Y_2\to Y_1$. 
By Lemma \ref{PBVB} $G_0$ acts on $E_{Y_2}\to Y_2$ faithfully via vector bundle automorphism.\\
The data $C, G_0, N_0,Y_2, W_2, Y_2\to W_2, E_{Y_2}$ satisfies the claim of the lemma. This finishes the proof.
\end{proof}

In the remaining of the section we apply Lemma \ref{YW} to build a vector bundle with the desired properties.

\begin{lem}
\label{Bundle1}
Let $d$ be a non-negative integer. There exists a constant $C=C(d)\in\mathbb{Z}^+$, only depending on $d$, such that, if 
\begin{itemize}
\item $G$ is a finite group,
\item$X$ and $Z$ are smooth projective complex varieties endowed with (not necessarily faithfully) $G$-actions via biregular automorphisms, 
\item $d=\dim X$ is the dimension of $X$, and $\dim X\geqq\dim Z$,
\item $f:X\to Z$ is a $G$-equivariant dominant morphism with rationally connected general fibres,
\item $N\unlhd G$ is a normal Abelian subgroup such that $N$ acts trivially on $Z$,
i.e. the $G$-action on $Z$ descends to a $G/N$-action on $Z$,
\item $E_X\to X$ is a vector bundle, on which $G$ acts faithfully via vector bundle automorphisms (in such a way that the $G$-action is compatible with the $G$-action on $X$), 
and $G$ respects the vector bundle structure,
%\item $e=\dim E_X$ is the dimension of the vector bundle,
\end{itemize}
then there exists
\begin{itemize}
\item a subgroup $G_0\leqq G$ such that its index is bounded by $C$, i.e. $|G:G_0|<C$,
\item a smooth projective complex variety $Y$, such that
\item $G_0$ acts on $Y$ via biregular automorphisms,
\item $N_0=G_0\cap N$ acts trivially on $Y$, 
i.e. the $G_0$-action on $Y$ descends to a $G_0/N_0$-action; moreover, there exists
\item a vector bundle $E_Y\to Y$, on which $G_0$ acts faithfully via vector bundle automorphisms (in such a way that the $G_0$-action is compatible with the $G_0$-action on $Y$), moreover
\item $\dim E_Y=\dim E_X$, i.e. the dimension of the total space of the vector bundle over $Y$ 
is equal to the dimension of the total space of the vector bundle over $X$.
\end{itemize}
\end{lem}

\begin{proof}
The lemma follows from Lemma \ref{YW} and induction on $d$.
\end{proof}

\begin{lem}
\label{Bundle2}
Let $d$ be a non-negative integer. There exists a constant $C=C(d)\in\mathbb{Z}^+$, only depending on $d$, such that, if 
\begin{itemize}
\item $X$ is a $d$ dimensional complex variety,
\item $G\leqq\Bir(X)$ is a finite subgroup of the birational automorphism group,  
\item $Z$ is a complex variety endowed a with a (not necessarily faithful) $G$-action via birational automorphisms,
\item $N\unlhd G$ is a normal Abelian subgroup such that $N$ acts trivially on $Z$, i.e. the $G$-action on $Z$ descends to a $G/N$-action,
\item $f:X\dashrightarrow Z$ is a $G$-equivariant dominant rational map with rationally connected general fibres,
\end{itemize}
then there exists
\begin{itemize}
\item a subgroup $G_0\leqq G$ such that its index is bounded by $C$, i.e. $|G:G_0|<C$,
\item a smooth projective complex variety $Y$, such that
\item $G_0$ acts on $Y$ via biregular automorphisms,
\item $N_0=G_0\cap N$ acts trivially on $Y$, 
i.e. the $G_0$-action on $Y$ descends to a $G_0/N_0$-action; moreover, there exists
\item a vector bundle $E_Y\to Y$, on which $G_0$ acts faithfully via vector bundle automorphisms (in such a way that the $G_0$-action is compatible with the $G_0$-action on $Y$), furthermore
\item $\dim E_Y=d=\dim X$,  i.e. the dimension of the total space of the vector bundle over $Y$ 
is equal to the dimension of $X$.
\end{itemize}
\end{lem}

\begin{proof}
Because of the lemma about smooth regularizations (Lemma \ref{reg}), 
it is enough to prove our lemma when $X$ and $Z$ are smooth projective varieties, $G\leqq\Aut(X)$ is a finite subgroup of the biregular automorphism group of $X$, 
$G$ acts on $Z$ via biregular automorphisms, and  $f$ is a dominant $G$-equivariant morphism. 
So we can assume these extra conditions. We setup $E_X\to X$ to be the vector bundle with zero dimensional fibres, i.e. $E_X\cong X$. Then we can apply Lemma \ref{Bundle1} to finish the proof.
\end{proof}

\begin{lem}
\label{MainLemma}
Let $d$ be a non-negative integer. There exists a constant $C=C(d)\in\mathbb{Z}^+$, only depending on $d$, such that, if 
\begin{itemize}
\item $X$ is a $d$ dimensional smooth projective complex variety,
\item $G\leqq\Bir(X)$ is a finite subgroup of the birational automorphism group,  
\item the pair $(\phi, Z)$ is the MRC fibration,
\item $N\unlhd G$ is a normal Abelian subgroup such that $N$ acts trivially on $Z$, i.e. the $G$-action on $Z$ descends to a $G/N$-action,
\end{itemize}
then there exists
\begin{itemize}
\item a subgroup $G_0\leqq G$ such that its index is bounded by $C$, i.e. $|G:G_0|<C$,
\item a smooth projective complex variety $Y$, such that
\item $G_0$ acts on $Y$ via biregular automorphisms,
\item $N_0=G_0\cap N$ acts trivially on $Y$, 
i.e. the $G_0$-action on $Y$ descends to a $G_0/N_0$-action; moreover, there exists
\item a vector bundle $E_Y\to Y$, on which $G_0$ acts faithfully via vector bundle automorphisms (in such a way that the $G_0$-action is compatible with the $G_0$-action on $Y$), furthermore
\item $\dim E_Y=d=\dim X$,  i.e. the dimension of the total space of the vector bundle over $Y$ 
is equal to the dimension of $X$.
\end{itemize}
\end{lem}

\begin{proof}
By the functoriality of the MRC fibration (see Corollary \ref{GMRC}), there is a natural induced $G$-action on $Z$ via birational automorphisms,
moreover $\phi$ is a $G$-equivariant dominant rational map with rationally connected fibres, therefore we can apply the previous lemma.
\end{proof}

\section{A Jordan type theorem on semilinear groups}
\label{JSL}
In this section we prove a Jordan type theorem on semilinear groups.

\begin{thm}
\label{groupmain}
Let $n$ and $m$ be positive integers. There exists a constant $J=J(n,m)\in\mathbb{Z}^+$, only depending on $n$ and $m$, such that if
\begin{itemize}
\item $k$  is a field of characteristic zero containing all roots of unity, 
\item $G$ is a finite subgroup of the semilinear group $\KL(n,k)\cong \GL(n)\rtimes \Aut k$, such that
\item every subgroup of $G$ can be generated by $m$ elements,
\item the image of the composite group homomorphism $G\hookrightarrow \KL(n,k)\twoheadrightarrow\Aut k$, denoted by $\Gamma$, is Abelian and fixes all roots of unity,
\end{itemize}
then there exists a subgroup  $H\leqq G$ with nilpotency class at most two and with index at most $J$. 
\end{thm}

Before proving the theorem we recall a slightly strengthened version of Jordan's original theorem on linear groups.
\begin{thm}
\label{Jor}
Let $n$ be a positive integer. There exists a constant $J=J(n)\in\mathbb{Z}^+$, only depending on $n$, 
such that if 
\begin{itemize}
\item $k$  is a field of characteristic zero, 
\item $G$ is a finite subgroup of the general linear group $\GL(n,k)$,
\end{itemize}
then there exists a characteristic  Abelian subgroup  $A\leqq G$ of index at most $J$.
\end{thm}

\begin{proof}
By Jordan's theorem on linear groups (Theorem 2.3 in \cite{Br11}) we can find an Abelian subgroup of bounded index. 
Than by Chermak-Delgado theory (Theorem \ref{CD}) we can find a characteristic Abelian subgroup of bounded index.
\end{proof}

\begin{proof}[Proof of Theorem \ref{groupmain}]
Fix $n, m, k, G$ and $\Gamma$ as in the theorem.
Consider the short exact sequence of groups given by
$$1\to N\to G\to\Gamma\to 1$$
where $N=\GL(n,k)\cap G$ and $\Gamma=\Imag(G\to\Aut k)$. By Theorem \ref{Jor}, $N$ contains a characteristic Abelian subgroup $A$ of index bound by $J_N=J_N(n)\in\mathbb{Z}^+$. 
Since $A$ is characteristic in $N$ and $N$ is normal in $G$, $A$ is a normal subgroup of $G$.\\
Consider the natural semilinear action of $G$ on the vector space $V=k^n$. Since $A$ is a finite Abelian subgroup of $\GL(V)$ and the ground field $k$ contains all roots of unity,
$A$ decomposes $V$ into common eigenspaces of its elements: $V=V_1\oplus V_2\oplus...\oplus V_r$ $(r\leqq n)$. 
As $A$ is normal in $G$, $G$ respects this decomposition, i.e. $G$ acts on the set of linear subspaces $\{V_1,V_2,...,V_r\}$ by permutations.
The kernel of this group action, denoted by $G_1$, is a bounded index subgroup of $G$ (indeed $|G:G_1|\leqq r!\leqq n!$). Furthermore, $A$ is central in $G_1$, i.e. $A\leqq \Z(G_1)$. 
To see this, notice that on an arbitrary fixed eigenspace $V_i$ $(1\leqq i\leqq r)$ $A$ acts by scalar matrices in such a way that all scalars are drawn from the set of the roots of unity. Since $G_1$ leaves $V_i$ invariant by definition and 
$\Imag(G_1\to\Aut k)$ fixes all roots of unity, our claim follows. After replacing $G$ with the bounded index subgroup $G_1$, we can assume that $A\leqq \Z(G)$.\\
As $A$ is a central subgroup of $G$, we can consider the quotient group $\overline{G}=G/A$. 
We only need to prove that $\overline{G}$ has a bounded index Abelian subgroup (Proposition \ref{CE}). 
Let $\overline{N}=N/A$, and consider the short exact sequence of groups
$$1\to \overline{N}\to\overline{G}\to\Gamma\to 1.$$ 
The number of elements of $\overline{N}$ is bounded by $J_N(n)$, by the definition of $A$, $\Gamma$ is Abelian, and every finite subgroup of $\overline{G}$ can be generated by $m$ elements, by the definition of $G$.
Hence we can apply Lemma \ref{DN2} to find a bounded index Abelian subgroup in $\overline{G}$, which finishes the proof.
\end{proof}

\section{Proof of the Main Theorem}
\label{PMT}

In this section, using the techniques developed in the previous sections, 
we prove our main theorem.

\begin{thm}
\label{VB}
Let $X$ be a smooth projective complex variety. Let the pair $(Z,\phi)$ be the MRC fibration of $X$ and $e=\dim X-\dim Z$ be the relative dimension.
There exists a constant $C=C(e, Z)$, only depending on $e$ and the birational class of $Z$, such that if
\begin{itemize}
\item $G\leqq\Bir(X)$ is a finite subgroup of the birational automorphism group,
\end{itemize}
then there exist
\begin{itemize}
\item a subgroup $G_0\leqq G$ of index bounded by $C$, i.e. $|G:G_0|<C$,
\item an Abelian normal subgroup $N_0\unlhd G_0$, such that
\item $G_0/N_0$ is Abelian, furthermore there exist
\item a smooth projective complex variety $Y$, such that
\item $G_0$ acts on $Y$ via biregular automorphisms,
\item $N_0$ acts trivially on $Y$, 
i.e. the $G_0$-action on $Y$ descends to the Abelian $G_0/N_0$-action, moreover, there exists
\item a vector bundle $E_Y\to Y$, on which $G_0$ acts faithfully via vector bundle automorphisms (in such a way that the $G_0$-action is compatible with the $G_0$-action on $Y$), and
\item $\dim E_Y=\dim X$,  i.e. the dimension of the total space of the vector bundle over $Y$ 
is equal to the dimension of $X$.
\end{itemize}
\end{thm}

\begin{proof}
Let $G\leqq\Bir(X)$ be an arbitrary finite subgroup of the birational automorphism group of the fixed complex variety $X$. 
By Theorem \ref{AbyA}, we can find a bounded index subgroup $G_1\leqq G$ and an Abelian normal subgroup $N_1\unlhd G_1$ of the bounded index subgroup, such that
$G_1/N_1$ is Abelian and $N_1$ acts trivially on $Z$. By applying Lemma \ref{MainLemma} to the data consisting of
the non-negative integer $d=\dim X$, the complex variety $X$ and the finite groups $G_1$ and $N_1$ we can finish the proof.
\end{proof}

\begin{cor}
\label{emb}
Let $X$ be a smooth projective complex variety. Let the pair $(Z,\phi)$ be the MRC fibration of $X$ and $e=\dim X-\dim Z$ be the relative dimension.
There exist constants $C=C(e, Z), m=m(e, Z)$, only depending on $e$ and the birational class of $Z$, such that if
\begin{itemize}
\item $G\leqq\Bir(X)$ is a finite subgroup of the birational automorphism group,
\end{itemize}
then there exist
\begin{itemize}
\item a subgroup $G_0\leqq G$ of index bounded by $C$, i.e. $|G:G_0|<C$, such that
\item every subgroup of $G_0$ can be generated by $m$ elements, moreover there exist
\item a function field $k$ of characteristic zero containing all roots of unity,
\item a non-negative integer $n$, such that $n\leqq \dim X$, and
\item an embedding of groups $G_0\hookrightarrow \KL(n,k)\cong \GL(n)\rtimes \Aut k$, such that
\item the image of the composite group homomorphism $G_0\hookrightarrow \KL(n,k)\twoheadrightarrow\Aut k$, denoted by $\Gamma$, is Abelian and fixes all roots of unity.
\end{itemize}
\end{cor}

\begin{proof}
Fix the complex variety $X$, and let $G\leqq\Bir(X)$ be an arbitrary finite subgroup of the birational automorphism group.  By Theorem \ref{m} the constant $m$ exists.
We can apply Theorem \ref{VB} to $X$ and $G$. From now on we use the notation of the theorem.
 Denote the generic fibre of the vector bundle $E_Y\to Y$ by $V$. $V$ is a $k(Y)$-vector space of dimension at most $\dim X$, and $G_0$ acts on $V$ faithfully via semilinear automorphisms.
Therefore $G_0$ embeds into  $\KL(V)\cong \GL(V)\rtimes \Aut k(Y)$. 
The group $\Gamma=\Imag(G_0\hookrightarrow \KL(V)\twoheadrightarrow\Aut k(Y))$ corresponds to the $G_0$-action on $Y$. 
Hence $\Gamma$ is Abelian, as the $G_0$-action on $Y$ descends to the action of the Abelian group $G_0/N_0$.  
Moreover $\Gamma$ fixes the field of the complex numbers,
in particular, it acts trivially on the set of the roots of unity. This finishes the proof. 
\end{proof}

\begin{thm}
\label{AlmostMain}
Let $X$ be a smooth projective complex variety. Let the pair $(Z,\phi)$ be the MRC fibration of $X$ and $e=\dim X-\dim Z$ be the relative dimension.
There exists a constant $J=J(e, Z)$, only depending on $e$ and the birational class of $Z$, such that if
\begin{itemize}
\item $G\leqq\Bir(X)$ is a finite subgroup of the birational automorphism group,
\end{itemize}
then there exists
\begin{itemize}
\item a nilpotent subgroup $G_0\leqq G$ of nilpotency class at most two and  of index at most $J$,
\end{itemize}
\end{thm}

\begin{proof}
By Corollary \ref{emb} we can embed a bounded index subgroup of $G$ into a semilinear group.
Than we can finish the proof by our theorem on semilinear groups (Theorem \ref{groupmain}).
\end{proof}

An immediate corollary is our main theorem.

\begin{proof}[Proof of the Theorem \ref{main}]
Let $X$ be an arbitrary variety over a field of characteristic zero. By Lemma \ref{C} we can assume that $X$ is a complex variety. 
We can also assume that $X$ is smooth and projective.
Let $(\phi, Z)$ be the MRC fibration and $e=\dim X-\dim Z$ be the relative dimension. Apply Theorem \ref{AlmostMain} to $X$
and observe that the relative dimension $e$ and the birational class of $Z$ only depend on the birational class of $X$.
This finishes the proof. 
\end{proof}

\end{document}